\documentclass[oneside]{amsart}

\usepackage[letterpaper,body={15.0cm,22.0cm}, mag=1000]{geometry}
\usepackage{amssymb}
\usepackage{amsthm}
\usepackage{amscd}

\numberwithin{equation}{section}
\theoremstyle{plain}

\newtheorem{thm}{Theorem}[section]
 
 \newtheorem{lemma}[thm]{Lemma}
\newtheorem{prop}[thm]{Proposition}

\newtheorem*{thma}{Theorem A}
\newtheorem*{thmb}{Theorem B}
\newtheorem*{thmc}{Theorem C}
\theoremstyle{definition}

\newcommand{\dlabel}[1]{\ifmmode \text{\ttfamily \upshape [#1] } \else
{\ttfamily \upshape [#1] }\fi \label{#1}}

\newcommand{\C}{\operatorname{C} }

\newcommand{\Z}{\operatorname{Z} }

\newcommand{\het}{\operatorname{ht}}

\newcommand{\gen}[1]{\left < #1 \right >}
\newcommand{\Aut}{\operatorname{Aut} }

\newcommand{\Hom}{\operatorname{Hom} }
\newcommand{\Inn}{\operatorname{Inn} }

\newcommand{\Autcent}{\operatorname{Autcent} }

\begin{document}

\setlength{\baselineskip}{15pt}

\title{On finite $p$-groups with abelian automorphism group}

\author{Vivek K. Jain}
\address{Department of Mathematics, Central University of Bihar, Patna  800 014, INDIA.}
\email{jaijinenedra@gmail.com}

\author{Pradeep K. Rai}
\address{School of Mathematics, Harish-Chandra Research Institute, Chhatnag Road, Jhunsi, Allahabad 211019, INDIA.}
\email{pradeeprai@hri.res.in}

\author{Manoj K. Yadav}
\address{School of Mathematics, Harish-Chandra Research Institute, Chhatnag Road, Jhunsi, Allahabad 211019, INDIA.}
\email{myadav@hri.res.in}


\begin{abstract}
We construct, for the first time, various types of specific non-special finite $p$-groups having abelian automorphism group. More specifically, we construct groups $G$ with abelian automorphism group such that $\gamma_2(G) < \Z(G) < \Phi(G)$, where $\gamma_2(G)$, $\Z(G)$ and $\Phi(G)$ denote the commutator subgroup, the center and the Frattini subgroup of $G$ respectively. For a finite $p$-group $G$ with elementary abelian automorphism group, we show that at least one of the following two conditions holds true: (i)  $\Z(G) = \Phi(G)$ is elementary abelian; (ii)  $\gamma_2(G) = \Phi(G)$ is elementary abelian, where $p$ is an odd prime.  We construct examples to show the existence of groups $G$ with  elementary abelian automorphism group for which exactly one of the above two conditions holds true. 

\end{abstract}

\maketitle

\section{Introduction}

Let $G$  be a finite group. An automorphism $\alpha$ of $G$ is called \emph{central} if $x^{-1}\alpha(x) \in \Z(G)$ for all $x \in G$, where $\Z(G)$ denotes the center of $G$. The set of all central automorphisms of $G$ is a normal subgroup of $\Aut(G)$, the group of all automorphisms of $G$. We denote this group by $\Autcent(G)$. Notice that $\Autcent(G) = \C_{\Aut(G)}(\Inn(G))$, the centralizer of $\Inn(G)$ in $\Aut(G)$, and $\Autcent(G) = \Aut(G)$ if $\Aut(G)$ is abelian.  We denote the commutator and Frattini subgroup of $G$ with $\gamma_2(G)$ and $\Phi(G)$, respectively.  Let  $G^{p^i} = \gen{x^{p^i} \mid x \in G}$ and $G_{p^i} = \gen{x \in G \mid x^{p^i} = 1}$, where $i \ge 1$ is an integer.  For finite abelian groups $H$ and $K$, $\Hom(H, K)$ denotes the group of all homomorphisms from $H$ to $K$. If $H$ is a subgroup (proper subgroup) of $G$, then we write $H \le G$ ($H < G$). A group $G$ is said to be \emph{purely non-abelian} if it does not have a non-trivial abelian direct factor. Throughout the paper, any unexplained $p$ always denotes an odd prime.

In this paper we construct, for the first time, various types of non-special finite $p$-groups $G$ such that $\Aut(G)$ is abelian.
 The story began in 1908 with the following question of H. Hilton \cite{pH08}: {\it Whether a non-abelian group can have an abelian group of isomorphisms (automorphisms)}. In 1913, G. A. Miller \cite{gM13}  constructed a non-abelian group $G$ of order $64$ such that $\Aut(G)$ is an elementary abelian group of order $128$. More examples of such $2$-groups were constructed in \cite{mC87, aJ02, rS82}.  For an odd prime $p$, the first example of a finite $p$-group $G$ such that $\Aut(G)$ is abelian was constructed by H. Heineken and H. Liebeck \cite{HL74} in 1974. In 1975, D. Jonah and M. Konvisser \cite{JK75} constructed $4$-generated groups of order $p^8$ such that  $\Aut(G)$ is an elementary abelian group of order $p^{16}$, where $p$ is any prime.  In 1975, by generalizing the constructions of Jonah and Konvisser, B. Earnley \cite[Section 4.2]{bE75} constructed $n$-generated special $p$-groups $G$ such that $\Aut(G)$ is abelian, where $n \ge 4$ is an integer and $p$ is any prime number. Among other things, Earnley also proved that there is no $p$-group $G$ of order $p^5$ or less such that $\Aut(G)$ is abelian. On the way to constructing finite $p$-groups of class $2$ such that all normal subgroups of $G$ are characteristic, in 1979 H. Heineken \cite{hH79} produced groups $G$ such that $\Aut(G)$ is abelian. In 1994, M. Morigi \cite{mM94} proved that there exists no group of order $p^6$ whose group of  automorphisms is abelian and constructed groups $G$ of order $p^{n^2 + 3n +3}$ such that $\Aut(G)$ is abelian, where $n$ is a positive integer. In particular, for $n=1$, it provides a group of order $p^7$ having an abelian automorphism group. 

There have also been attempts to get structural information of finite groups having abelian automorphism group. In 1927, 
C. Hopkins \cite{cH27}, among other things, proved that a finite $p$-group $G$ such that $\Aut(G)$ is abelian, can not have a non-trivial abelian direct factor. In 1995, M. Morigi \cite{mM95} proved that the minimal number of generators for a $p$-group with abelian automorphism group is $4$. In 1995, P. Hegarty \cite{pH95} proved that if $G$ is a non-abelian $p$-group such that $\Aut(G)$ is abelian, then $|\Aut(G)| \ge p^{12}$, and the minimum is obtained by the group of order $p^7$ constructed by M. Morigi. Moreover, in 1998, G. Ban and S. Yu \cite{BY98} obtained independently the same result and proved that if $G$ is a group of order $p^7$ such that  $\Aut(G)$ is abelian, then $|\Aut(G)| = p^{12}$.

We remark here that all the examples (for an odd prime $p$) mentioned above are special $p$-groups. 
In 2008, A. Mahalanobis \cite{aM08} published the following conjecture: {\it For an odd prime $p$, any finite $p$-group having abelian automorphism group is special}. The first and third authors \cite{JY12} provided counter examples to this conjecture by constructing a class of  non-special finite $p$-groups $G$ such that $\Aut(G)$ is abelian. These counter examples, constructed in \cite{JY12}, enjoy the following properties: (i) $|G| = p^{n+5}$, where $p$ is an odd prime and $n$ is an integer $\geq 3$; (ii) $\gamma_2(G)$ is a proper subgroup of $\Z(G) = \Phi(G)$; (iii) exponents of $\Z(G)$ and $G/\gamma_2(G)$ are same and it is equal to $p^{n-1}$; (iv) $\Aut(G)$ is abelian of exponent $p^{n-1}$. 

Having all this information in hands, one might expect that some weaker form of the conjecture of Mahalanobis  still holds true. Two obvious weaker forms of the conjecture are: (WC1) For a finite $p$-group $G$ with $\Aut(G)$ abelian, $\Z(G) = \Phi(G)$ always holds true; (WC2) For a finite $p$-group $G$ with $\Aut(G)$ abelian and $\Z(G) \neq \Phi(G)$, $\gamma_2(G) = \Z(G)$ always holds true. So, on the way to exploring some general structure on the class of such groups $G$, it is natural to ask the following question:
\vspace{.1in}

\noindent{\bf Question.} Does there exist a finite $p$-group $G$ such that $\gamma_2(G) \le \Z(G) < \Phi(G)$ and $\Aut(G)$ is abelian?
\vspace{.1in}

Disproving (WC1) and (WC2), we provide affirmative answer to this question,  in the following theorem, which we prove in Section 3:

\begin{thma}
For every positive integer $n \ge 4$ and every odd prime $p$, there exists a group $G$ of order $p^{n+10}$ and  exponent $p^n$ such that 
\begin{enumerate}
 \item for $n = 4$, $\gamma_2(G) = \Z(G) < \Phi(G)$ and $\Aut(G)$ is abelian;
\item for $n \ge 5$, $\gamma_2(G) < \Z(G) < \Phi(G)$ and $\Aut(G)$ is abelian.
\end{enumerate}
Moreover, the order of $\Aut(G)$ is $p^{n+20}$.
\end{thma}

One more weaker form of the above said conjecture is: (WC3) If $\Aut(G)$ is an elementary abelian $p$-group, then $G$ is special. 
As remarked above, all $p$-groups $G$ (except the ones in \cite{JY12}) available in the literature and having abelian automorphism group are special $p$-groups. Thus it follows that $\Aut(G)$, for all such groups $G$, is elementary abelian. Y. Berkovich and Z. Janko \cite[Problem 722]{BJ08} published the following long standing problem: {\it (Old problem) Study the $p$-groups $G$ with elementary abelian $\Aut(G)$}.

Let $G$ be an arbitrary finite $p$-group such that $\Aut(G)$ is elementary abelian. Then it follows from Theorem \ref{thm1BL} (proved in Section 4 below) that one of the following two conditions necessarily holds true: (C1) $\Z(G) = \Phi(G)$ is elementary abelian; (C2) $\gamma_2(G) = \Phi(G)$ is elementary abelian. {\it So one might expect that for such groups $G$ both of the conditions (C1) and (C2) hold true, i.e., WC(3) holds true, or, a little less ambitiously, (C1) always holds true or (C2) always holds true}. In the following two theorems, which we prove in Section 4, we show that none of the statements in the preceding sentence holds true. 

\begin{thmb}
 There exists a group $G$ of order $p^9$ such that $\Aut(G)$ is elementary abelian of order $p^{20}$, $\Phi(G) < \Z(G)$ and $\gamma_2(G) = \Phi(G)$ is elementary abelian. 
\end{thmb}

\begin{thmc}
 There exists a group $G$ of order $p^8$ such that $\Aut(G)$ is elementary abelian of order $p^{16}$, $\gamma_2(G) < \Phi(G)$ and $\Z(G) = \Phi(G)$ is elementary abelian.
\end{thmc}

Now we review  non-special $2$-groups having abelian automorphism group. In contrast to $p$-groups for odd primes, there do exist finite $2$-groups $G$ with $\Aut(G)$ abelian and $G$ satisfies either of the following two properties: (P1) $G$ is $3$-generated; (P2) $G$ has a non-trivial abelian direct factor. The first $2$-group having abelian automorphism group was constructed by G. A. Miller \cite{gM13} in 1913. This is a $3$-generated group and, as mentioned above, it has order $64$ with elementary abelian automorphism group of order $128$. B. Earnley \cite{bE75} showed that there are two more groups of order $64$ having elementary abelian automorphism group. These groups are also $3$-generated. Further B. Earnley \cite[Theorem 2.3]{bE75} gave a complete description of $2$-groups satisfying (P2) and having abelian automorphism group and established the existence of such groups. 

Let $G$ be a purely non-abelian finite $2$-group such that $\Aut(G)$ is elementary abelian. Thus $\Aut(G) = \Autcent(G)$. Then $G$ satisfies one of the three conditions of Theorem \ref{thmJ1}. We here record that there exist groups $G$ which satisfy exactly one condition of this theorem. It is easy to show that the $2$-group $G_2$ constructed in \eqref{group2} satisfies only the first condition of Theorem \ref{thmJ1} and  the $2$-group $G_3$ constructed in \eqref{group3} satisfies only the second condition of Theorem \ref{thmJ1}. That $\Aut(G_2)$ and $\Aut(G_3)$ are elementary abelian, can be checked using GAP \cite{gap}. The examples of $2$-groups $G$ satisfying only the third condition of Theorem \ref{thmJ1} with $\Aut(G)$ elementary abelian were constructed by A. Miller \cite{gM13} and M. J. Curran \cite{mC87}. 

The examples constructed in Theorems A, B and C indicate that it is difficult to put an obvious structure on the class of groups $G$ such that $\Aut(G)$ is abelian or even elementary abelian. We hope that this paper will be helpful in giving a direction to the research area under consideration. We remark that many non-isomorphic groups, satisfying the conditions of the above theorems, can be obtained by making suitable changes in the presentations given in \eqref{group1}, \eqref{group2} and \eqref{group3} below. We conclude this section with a further remark that the kind of examples constructed here may be useful in cryptography (see \cite{aM08} for more details).

\section{Some prerequisites and basic results}

We start with the following two theorems of M. H. Jafari \cite{mJ06}.

\begin{thm}[Theorem 3.4]\label{thmJ}
 Let $G$ be a finite purely nonabelian $p$-group, p odd, then $\Autcent(G)$ is an elementary abelian $p$-group if and only if  the exponent of $Z(G)$ is $p$ or exponent of $G/\gamma_2(G)$ is $p$.
\end{thm}

A finite abelian $p$-group $G$ is said to be \emph{ce-group} if $G$ can be written as a direct product of a cyclic group $A$ of order $p^n$, $n > 1$ and an elementary abelian $p$-group $B$.

\begin{thm}[Theorem 3.5]\label{thmJ1}
Let G be a purely nonabelian $2$-group. Then $\Autcent(G)$ is elementary
abelian if and only if one of the following conditions holds:
\begin{enumerate}
\item the exponent of $G/\gamma_2(G)$ is $2$;
\item the exponent of $Z(G)$ is $2$;
\item the greatest common divisor of the exponents of $G/\gamma_2(G)$ and $Z(G)$ is $4$ and $G/\gamma_2(G)$, $Z(G)$ are ce-groups having the properties that an elementary part of $Z(G)$ is contained in $\gamma_2(G)$ and there exists an
element $z$ of order $4$ in a cyclic part of $Z(G)$ with $z\gamma_2(G)$ lying in a cyclic part of $G/\gamma_2(G)$
such that twice of the order of $z\gamma_2(G)$ is equal to the exponent of $G/\gamma_2(G)$.
\end{enumerate}
\end{thm}

The next result is by B. Earnley \cite[Corollary 3.3]{bE75}.

\begin{thm}\label{thmE}
 Let $G$ be a non-abelian finite $p$-group of exponent $p$, where $p$ is an odd prime. Then $\Aut(G)$ is non-abelian.
\end{thm}

The following two results are well known.

\begin{lemma}\label{lemma1a}
 Let $A$, $B$ and $C$ be finite abelian groups. Then $\Hom(A \times B, \;C) \cong \Hom(A, \; C) \times \Hom(B, \;C)$ and $\Hom(A, \;B \times C) \cong \Hom(A, \; B) \times \Hom(A, \;C)$.
\end{lemma}

\begin{lemma}\label{lemma1b}
 Let $\C_r$ and $\C_s$ be two cyclic groups of order $r$ and $s$ respectively. Then $\Hom(\C_r, $ $\C_s) \cong \C_d$, where $d$ is the greatest common divisor of $r$ and $s$.
\end{lemma}

The following result follows from \cite[Theorem 1]{AY65}.

\begin{prop}\label{prop1}
 Let $G$ be a purely non-abelian finite $p$-group. Then $|\Autcent(G)| = |\Hom(G$ $/\gamma_2(G), \Z(G))|$.
\end{prop}

Let $A$ be an abelian $p$-group and $a \in A$. For a positive integer $n$, $p^n$  is said to be the \emph{height} of $a$ in $A$, denoted by $\het(a)$, if $a \in A^{p^n}$ but
$a \not\in  A^{p^{n+1}}$. Let $H$ be a  $p$-group of class $2$. We denote the exponents of $\Z(H)$, $\gamma_2(H)$, $H/\gamma_2(H)$ by $p^a$, $p^b$, $p^c$ respectively and $d = \text{min}(a, c)$. We define $R := \{z \in \Z(H) \mid \; |z| \le p^d\}$ and $K := \{x \in H \mid \het(x\gamma_2(H)) \ge p^b\}$. Notice that $K = H^{p^b}\gamma_2(H)$. 
Now we state the following important result of J. E. Adney and T. Yen \cite[Theorem 4]{AY65} (in our notations).

\begin{thm}\label{thmAY}
Let $H$ be a  purely non-abelian $p$-group of class $2$, $p$ odd, and let $H/\gamma_2(H) = \Pi_{i = 1}^n\gen{x_i\gamma_2(H)}$. Then $\Autcent(H)$ is abelian if and only if

\text{(i)} $R = K$, and 

\text{(ii)} either $d = b$ or $d > b$ and $R/\gamma_2(H) = \gen{x_1^{p^b}\gamma_2(H)}$.
\end{thm}

Let $G$ be a finite $p$-group of nilpotency class $2$ generated by $x_1, x_2, \dots, x_d$, where $d$ is a positive integer.
Let
$e_{x_i} = x_1^{a_{i1}} x_2^{a_{i2}} \cdots x_d^{a_{id}} = $  {\tiny$\prod$}$_{j=1}^d x_j^{a_{ij}}$, where $x_i \in G$ and $a_{ij}$ are non-negative integers for
 $1 \le i, j \le d$. Since the nilpotency class of $G$ is $2$, we have  
\begin{equation}\label{0a}
[x_k, e_{x_i}]  =  [x_k,  \text{\tiny$\prod$}_{j=1}^d x_j^{a_{ij}}] = \text{\tiny$\prod$}_{j=1}^d [x_k, x_j^{a_{ij}}] = \text{\tiny$\prod$}_{j=1}^d [x_k, x_j]^{a_{ij}}
\end{equation}
and 
\begin{eqnarray}\label{0b}
 [e_{x_k}, e_{x_i}]  &=&  [\text{\tiny$\prod$}_{l=1}^d x_l^{a_{kl}},  \text{\tiny$\prod$}_{j=1}^d x_j^{a_{ij}}] = \text{\tiny$\prod$}_{j=1}^d \text{\tiny$\prod$}_{l=1}^d [x_l^{a_{kl}}, x_j^{a_{ij}}]\\
& = & \text{\tiny$\prod$}_{j=1}^d \text{\tiny$\prod$}_{l=1}^d [x_l, x_j]^{a_{kl} a_{ij}}.\nonumber
\end{eqnarray}

Equations \eqref{0a} and \eqref{0b} will be used for our calculations without any further reference.

\section{Groups $G$ with $\Aut(G)$ abelian}

In this section we construct a class of finite $p$-groups such that  $\gamma_2(G) \le \Z(G) < \Phi(G)$ and $\Aut(G)$ is abelian.
Let $n \ge 4$ be a positive integer and $p$ be an odd prime. Consider the following group:

\begin{eqnarray}\label{group1}
G &=& \Big\langle x_1,\; x_2, \;x_3, \;x_4  \mid x_1^{p^n} = x_2^{p^4} = x_3^{p^4} = x_4^{p^2} = 1, \; [x_1, x_2] = x_2^{p^2},\\
& &  [x_1, x_3] = x_2^{p^2},\;  [x_1, x_4] = x_3^{p^2},\; [x_2, x_3] = x_1^{p^{n-2}},\; [x_2, x_4] = x_3^{p^2},\nonumber\\
& & [x_3, x_4] = x_2^{p^2} \Big\rangle. \nonumber
\end{eqnarray}

Throughout this section, $G$ always denotes the group given in \eqref{group1}. It is easy to see that $G$ enjoys the properties given in the following lemma.
\begin{lemma}\label{lemma1}
 The group $G$ is a regular $p$-group of nilpotency class $2$ having order $p^{n+10}$ and exponent $p^n$. For $n = 4$, $\gamma_2(G) = \Z(G) < \Phi(G)$ and  for $n \ge 5$, $\gamma_2(G) < \Z(G) < \Phi(G)$.  
\end{lemma}

Let
$e_{x_i} = x_1^{a_{i1}} x_2^{a_{i2}}x_3^{a_{i3}}x_4^{a_{i4}} =$  {\tiny$\prod$}$_{j=1}^4 x_j^{a_{ij}}$, where $x_i \in G$ and $a_{ij}$ are non-negative integers for  $1 \le i, j \le 4$.
Let $\alpha$ be an automorphism of $G$. Since the nilpotency class of $G$ is $2$ and $\gamma_2(G)$ is generated by $x_1^{p^{n-2}}$, $x_2^{p^2}$, $x_3^{p^2}$, we can write $\alpha(x_i) = x_i e_{x_i} = x_i ${\tiny$\prod$}$_{j=1}^4 x_j^{a_{ij}}$ for some non-negative integers $a_{ij}$ for $1 \le i, j \le 4$.

\begin{prop}
Let $G$ be the group defined in \eqref{group1} and $\alpha$ be an automorphism of $G$ such that $\alpha(x_i) = x_i e_{x_i} = x_i ${\tiny$\prod$}$_{j=1}^4 x_j^{a_{ij}}$, where $a_{ij}$ are some non-negative integers for $1 \le i, j \le 4$. Then the following equations hold:
\begin{eqnarray}
 && a_{31} \equiv 0 \mod{p^{n-2}},\label{e00}\\
&& -a_{32}+a_{13}+a_{13}a_{44}\equiv 0 \mod{ p^2},\label{e7}\\
&&-a_{33}+a_{44}+a_{11}+a_{11}a_{44}+a_{12}+a_{12}a_{44}\equiv 0 \mod{ p^2},\label{e8}\\
 && a_{21} \equiv 0 \mod{p^{n-2}}, \label{e0}\\
&& a_{44}-a_{22}+a_{33}+a_{33}a_{44}\equiv 0 \mod{ p^2},\label{e5}\\
&& a_{32}-a_{23}+a_{32}a_{44} \equiv 0 \mod{ p^2},\label{e6}\\
&&-a_{13}+a_{12}a_{23}-a_{13}a_{22}\equiv 0 \mod{ p^2},\label{e11}\\
&& a_{23}+ a_{11}+a_{11}a_{22}+a_{11}a_{23}+a_{13}a_{24}-a_{14}a_{23} \equiv 0 \mod{ p^2},\label{e1}\\
&& -a_{23}+a_{24}+a_{11}a_{24}-a_{14}+a_{12}a_{24}-a_{14}a_{22} \equiv 0 \mod{ p^2},\label{e2}\\
&& a_{12}+a_{12}a_{33}-a_{13}a_{32}\equiv 0 \mod{ p^2},\label{e12}\\
&& a_{32}+a_{33}+a_{11}-a_{22}-a_{14}+a_{11}a_{32}+a_{11}a_{33}+a_{13}a_{34}-a_{14}a_{33} \label{e3}\\
&& \equiv 0 \mod{ p^2},\nonumber\\
&& a_{34}+a_{11}a_{34}+a_{12}a_{34}-a_{14}a_{32}-a_{23}\equiv 0 \mod{ p^2},\label{e4}\\
&&a_{23}-a_{32}+a_{23}a_{44}\equiv 0 \mod{ p^2},\label{e9}\\
&&-a_{33}+a_{44}+a_{22}+a_{22}a_{44} \equiv 0 \mod{ p^2},\label{e10}\\
&&-a_{11}+a_{33}+a_{22}+a_{22}a_{33}-a_{23}a_{32} \equiv 0 \mod{ p^2}.\label{e13}
\end{eqnarray}
\end{prop}
\begin{proof}
Let $\alpha$ be the automorphism of $G$ such that $\alpha(x_i) = x_i e_{x_i}$, $1 \le i \le 4$ as defined above.
Since $G_{p^2}=\langle x_1^{p^{n-2}}, x_2^{p^2}, x_3^{p^2}, x_4 \rangle$ is a characteristic subgroup of $G$,  $\alpha (x_4) \in G_{p^2}$. 
Thus we get the following set of equations:
\begin{eqnarray}
a_{41} &\equiv 0& \mod{p^{n-2}},\label{e15} \\
a_{4i} &\equiv 0& \mod{p^{2}}, ~\text{for}~i = 2,\;3.\label{e16} 
\end{eqnarray}

We prove equations \eqref{e00} - \eqref{e8} by comparing the powers of $x_i$'s in $\alpha([x_1, x_4]) = \alpha(x_3^{p^2})$. 

\begin{eqnarray*}
\alpha([x_1, x_4]) &=&  [\alpha(x_1), \alpha(x_4)] = [x_1 e_{x_1}, x_4 e_{x_4}] \\
&=& [x_1, x_4] [x_1, e_{x_4}] [e_{x_1}, x_4] [e_{x_1}, e_{x_4}]\\
& = & [x_1, x_4] \text{\tiny$\prod$}_{j=1}^4 [x_1, x_j]^{a_{4j}} \text{\tiny$\prod$}_{j=1}^4 [x_4, x_j]^{-a_{1j}}  
\text{\tiny$\prod$}_{j=1}^4 \text{\tiny$\prod$}_{l=1}^4 [x_l, x_j]^{a_{1l} a_{4j}}\\
& = & [x_1, x_2]^{a_{42}  + a_{11}a_{42} - a_{12}a_{41}}
       [x_1, x_3]^{a_{43} + a_{11} a_{43} - a_{13} a_{41}}\\
& &   [x_1, x_4]^{1 + a_{44} + a_{11} + a_{11}a_{44} - a_{14}a_{41}} 
       [x_2, x_3]^{a_{12} a_{43} - a_{13} a_{42}}\\
 & &   [x_2, x_4]^{a_{12} + a_{12} a_{44} - a_{14} a_{42}}
       [x_3, x_4]^{a_{13} + a_{13} a_{44} - a_{14} a_{43}}\\
& = &  x_1^{p^{n-2}(a_{12} a_{43} - a_{13} a_{42})}\\
& &  x_2^{p^2(a_{42} + a_{43} + a_{13} + a_{11}a_{42} - a_{12}a_{41} + a_{11} a_{43} - a_{13} a_{41} + a_{13} a_{44} - a_{14} a_{43})}\\
& & x_3^{p^2(1+a_{44} +  a_{11} + a_{12} + a_{11} a_{44} - a_{14} a_{41} + a_{12} a_{44} - a_{14} a_{42})}.
\end{eqnarray*}

On the other hand
\[\alpha([x_1, x_4]) = \alpha(x_3^{p^2}) = x_3^{p^2} x_1^{p^2a_{31}} x_2^{p^2a_{32}}x_3^{p^2a_{33}}x_4^{p^2a_{24}} = 
x_1^{p^2a_{31}} x_2^{p^2a_{32}}x_3^{p^2(1+ a_{33})}.\]

Comparing the powers of $x_1$ and using \eqref{e16}, we get $a_{31} \equiv 0 \mod{ p^{n-2}}$.
Comparing the powers of $x_2$ and $x_3$, and using  \eqref{e15} - \eqref{e16}, we get
\begin{eqnarray*}
 && -a_{32}+a_{13}+a_{13}a_{44}\equiv 0 \mod{ p^2},\\
&&-a_{33}+a_{44}+a_{11}+a_{11}a_{44}+a_{12}+a_{12}a_{44}\equiv 0 \mod{ p^2}.
\end{eqnarray*}

Hence equations \eqref{e00} - \eqref{e8} hold.

Equations \eqref{e0} - \eqref{e6} are obtained by comparing the powers of $x_1$, $x_2$ and $x_3$ in $\alpha([x_3, x_4]) = \alpha (x_2^{p^2})$ and using equations \eqref{e00}, \eqref{e15} and \eqref{e16}. Equations \eqref{e11} - \eqref{e2}
are obtained by comparing the powers of $x_1$, $x_2$ and $x_3$ in $\alpha([x_1, x_2]) = \alpha (x_2^{p^2})$ and using equation
 \eqref{e0}. Equations \eqref{e12} - \eqref{e4} are obtained by comparing the powers of $x_1$, $x_2$ and $x_3$ in $\alpha([x_1, x_3]) = \alpha (x_2^{p^2})$ and using equations \eqref{e00} and \eqref{e0}. Equations \eqref{e9} - \eqref{e10}
are obtained by comparing the powers of  $x_2$ and $x_3$ in $\alpha([x_2, x_4]) = \alpha (x_3^{p^2})$ and using equations \eqref{e0}, \eqref{e15} and \eqref{e16}. The last equation \eqref{e13} is obtained by comparing the powers of $x_1$ in  $\alpha([x_2, x_3]) = \alpha (x_1^{p^{n-2}})$.
\end{proof}

\begin{thm}\label{thm1}
  Let $G$ be the group defined in \eqref{group1}. Then all automorphisms of $G$ are central. 
\end{thm}
\begin{proof}
 We start with the claim that $1 + a_{44} \not\equiv 0 \mod{p}$. For, let us assume the contrary, i.e., $p$ divides $1 +a_{44}$. Then
\[\alpha(x_4^p) =  \alpha(x_4)^p = x_4^{p(1+a_{44})}(x_1^{a_{41}}x_2^{a_{42}}x_3^{a_{43}})^p \in \Z(G),\]
since $a_{4j} \equiv 0 \mod{p^2}$ for $1 \le j \le 3$ by equations \eqref{e15} and \eqref{e16}. But this is not possible as $x_4^p \not\in
\Z(G)$. This proves our claim. Subtracting \eqref{e9} from \eqref{e7}, we get $(1 + a_{44})(a_{13} - a_{23}) \equiv 0 \mod{p^2}$. Since $p$ does not divide $1+a_{44}$, we get
\begin{equation}
 a_{13} \equiv a_{23} \mod{p^2}.\label{e17}
\end{equation}
By equations \eqref{e11} and \eqref{e17} we have
\begin{equation}
 a_{13}(1 - a_{12} + a_{22}) \equiv 0 \mod{p^2}.\label{e18}
\end{equation}
Here we have three possibilities, namely (i) $a_{13} \equiv 0 \mod{p^2}$, (ii) $a_{13} \equiv 0 \mod{p}$, but $a_{13} \not\equiv 0 \mod{p^2}$, (iii) $a_{13} \not\equiv 0 \mod{p}$. We are going to show that cases (ii) and (iii) do not occur and in the case (i) $a_{ij} \equiv 0 \mod{p^2}$, $1 \le i, j \le 4$.

\noindent {\bf Case (i).} Assume that $a_{13} \equiv 0 \mod{p^2}$. Equations \eqref{e6} and \eqref{e17}, together with the fact that $p$ does not divide $1+a_{44}$, gives $a_{32} \equiv 0 \mod{p^2}$. We claim that $1 + a_{33} \not\equiv 0 \mod{p}$. Suppose $p$ divides $1 + a_{33}$. Since $a_{32} \equiv 0 \mod{p^2}$ and $a_{31} \equiv 0 \mod{p^{n-2}}$ (equation \eqref{e00}), we get $\alpha(x_3^{p^3}) = x_1^{p^3a_{31}}x_2^{p^3a_{32}}x_3^{p^3(1+a_{33})}x_4^{p^3a_{34}} = 1$, which is not possible. This proves our claim. So by equation \eqref{e12}, we get $a_{12} \equiv 0 \mod{p^2}$.

Subtracting  \eqref{e10} from \eqref{e8}, we get $(a_{11} - a_{22})(1 + a_{44}) \equiv 0 \mod{p^2}$. This implies that $a_{11} - a_{22} \equiv 0 \mod{p^2}$. Since $a_{i3} \equiv 0 \mod{p^2}$ for $i = 1, 2$, by equation \eqref{e1} we get $a_{11}(1 + a_{11}) \equiv 0 \mod{p^2}$. Thus $p^2$ divides  $a_{11}$ or $1+a_{11}$. We claim that $p^2$ can not divide  $1+a_{11}$. For, suppose the contrary, i. e., 
 $a_{11} \equiv -1~ \mod{p^2}$.  Since $n-2 \geq 2$ and $a_{12} \equiv a_{13} \equiv 0 \mod{p^2}$, we get
\[\alpha(x_1)^{p^{n-2}} = x_1^{p^{n-2}(1+a_{11})}x_2^{p^{n-2}a_{12}}x_3^{p^{n-2}a_{13}}x_4^{p^{n-2}a_{14}}=1.\]
This contradiction, to the fact that order of $x_1$ is $p^n$, proves our claim. Hence $p^2$ divides $a_{11}$. Since  $a_{11} - a_{22} \equiv 0 \mod{p^2}$, by equation \eqref{e10}, it follows that $a_{33} \equiv a_{44}  \mod{p^2}$. Putting the values $a_{23}$, $a_{11}$ and $a_{22}$ in \eqref{e13}, we get $a_{33} \equiv 0 \mod{p^2}$. Thus $a_{44} \equiv 0 \mod{p^2}$. Putting values of $a_{32}$, $a_{33}$, $a_{11}$, $a_{22}$ and $a_{13}$ in \eqref{e3}, we get $a_{14} \equiv 0 \mod{p^2}$. Putting values of $a_{12}$, $a_{14}$, $a_{11}$ and $a_{23}$ in \eqref{e4}, we get $a_{34} \equiv 0 \mod{p^2}$. Putting above values in \eqref{e2}, we get $a_{24} \equiv 0 \mod{p^2}$. Hence $a_{ij} \equiv 0 \mod{p^2}$ for $1 \le i, j \le 4$.

\noindent {\bf Case (ii).} Assume that $a_{13} \equiv 0 \mod{p}$, but $a_{13} \not\equiv 0 \mod{p^2}$.
Equation \eqref{e18} implies that $(1 - a_{12} + a_{22}) \equiv 0 \mod{p}$. Now consider all the equations \eqref{e0}-\eqref{e13} $\mod{p}$. Repeating the arguments of Case (i) after replacing $p^2$ by $p$,
we get the following facts: (a) $a_{32} \equiv 0 \mod{p}$ (by \eqref{e6}); (b) $a_{12} \equiv 0 \mod{p}$ (by \eqref{e12}); (c) $a_{11} - a_{22} \equiv 0 \mod{p}$ (subtracting  \eqref{e10} from \eqref{e8}); (d) $a_{11}(1 + a_{11}) \equiv 0 \mod{p}$ (by \eqref{e1}). We claim that $a_{11} \equiv 0 \mod{p}$. For, suppose that $a_{11} +1 \equiv 0 \mod{p}$. Since $n-1 \geq 3$ and $a_{12} \equiv a_{13} \equiv 0 \mod{p}$, it follows that $\alpha(x_1)^{p^{n-1}} = x_1^{p^{n-1}(1+a_{11})}x_2^{p^{n-1}a_{12}}x_3^{p^{n-1}a_{13}}x_4^{p^{n-1}a_{14}} = 1$, which is a contradiction. This proves that $p$ can not divide $a_{11}+1$. Hence $p$ divides $a_{11}$, and therefore by fact (c), we have $a_{22} \equiv 0 \mod{p}$. This gives a contradiction to the fact that $(1 - a_{12} + a_{22}) \equiv 0 \mod{p}$. Thus Case (ii) does not occur.

\noindent {\bf Case (iii).} Finally assume that $a_{13} \not\equiv 0 \mod{p}$. Thus $(1 - a_{12} + a_{22}) \equiv 0 \mod{p^2}$, i.e., $1+a_{22} \equiv a_{12} \mod{p^2} $ (we'll use this information throughout the remaining proof without referring).
Notice that $(\alpha(x_2x_1^{-1}))^{p^2}=x_1^{-p^2(1+a_{11})}$. Since the order of $(\alpha(x_2x_1^{-1}))^{p^2}$ is $ p^{n-2}$, $p$ does not divide $(1+a_{11})$.  Putting the value of $a_{32}$ from \eqref{e9} into \eqref{e6}, we have $a_{23}=a_{23}{(1+a_{44})^2} \mod {p^2}$. Since $a_{23} \equiv a_{13} \mod{p^2}$ (equation \eqref{e17}) and $a_{13} \not\equiv 0 \mod{p}$, it follows that $a_{23} \not \equiv 0 \mod{p}$. Hence $(1+a_{44})^2\equiv 1 \mod {p^2}$. This gives $a_{44}(a_{44}+2) \equiv 0 \mod{p^2}$. Thus we have three cases (iii)(a) $ a_{44} \equiv 0 \mod{p^2}$, (iii)(b) $a_{44} \equiv 0 \mod{p}$, but $a_{44} \not\equiv 0 \mod{p^2}$ and (iii)(c) $a_{44} \not\equiv 0 \mod{p}$. We are going to consider these cases one by one.

\noindent {\bf Case (iii)(a).} Suppose that $ a_{44} \equiv 0 \mod{p^2}$. Using this in \eqref{e9} and \eqref{e10}, we get $a_{32} \equiv a_{23} \mod{p^2}$ and   $a_{22} \equiv a_{33} \mod{p^2}$ respectively.    
Putting the value of $a_{44}$ in \eqref{e8}, we have $a_{12}+a_{11} \equiv a_{33} \mod{p^2}$. Further, replacing $a_{12} $ by $1+a_{22}$ and $a_{22}$ by $a_{33}$, we have $1+a_{11} \equiv 0 \mod{p^2}$, which is a contradiction.

\noindent {\bf Case (iii)(b).} Suppose that $a_{44} \equiv 0 \mod{p}$, but $a_{44} \not\equiv 0 \mod{p^2}$. 
Notice that by reading the equations $\mod{p}$, arguments of Case (iii)(a) show that $1 + a_{11} \equiv 0 \mod{p}$, which is again a contradiction.
 
\noindent {\bf Case (iii)(c).} Suppose that $a_{44} \not\equiv 0 \mod{p}$. This implies that $a_{44} \equiv -2 \mod{p^2}$.
 Putting this value of $a_{44}$ in the difference of \eqref{e5} and \eqref{e8}, we get
   $a_{11}+a_{12}-a_{22} \equiv 0 \mod{p^2}$. Since $1+a_{22} \equiv a_{12} \mod{p^2}$, this equation contradicts the fact that $1+a_{11} \not\equiv 0 \mod{p}$.

Thus Case (iii) can not occur. This completes the proof of the theorem. \hfill $\Box$

\end{proof}

Now we are ready to prove Theorem A stated in the introduction.\\

\noindent {\bf Proof of Theorem A.}
Let $G$ be the group defined in \eqref{group1}. By  Lemma \ref{lemma1}, we have  $|G| = p^{n+10}$, $\gamma_2(G) = \Z(G) < \Phi(G)$ for $n = 4$ and   $\gamma_2(G) < \Z(G) < \Phi(G)$ for $n \ge 5$. By Theorem \ref{thm1}, we have $\Aut(G) = \Autcent(G)$. Thus to complete the proof of the theorem, it is sufficient to prove that $\Autcent(G)$ is an abelian group. Since $\Z(G) < \Phi(G)$, $G$ is purely non-abelian. The exponents of $\Z(G)$, $\gamma_2(G)$ and $G/\gamma_2(G)$ are $p^{n-2}$, $p^2$ and $p^{n-2}$ respectively. Thus we get  
\[R =\{z \in \Z(G) \mid \; |z| \le p^{n-2}\} = \Z(G)\]
and 
\[K =  \{x \in G \mid \het(x\gamma_2(G)) \ge p^2\} = G^{p^2}\gamma_2(G) = \Z(G).\]
This shows that $R = K$. Also $R/ \gamma_2(G) = \Z(G)/ \gamma_2(G) = \gen{x_1^{p^2}\gamma_2(G)}$. Thus all the conditions of Theorem \ref{thmAY} are now satisfied. Hence $\Autcent(G)$ is abelian. That the order of $\Aut(G)$ is $p^{n+20}$ can be easily proved by using Lemmas \ref{lemma1a}, \ref{lemma1b}, Proposition \ref{prop1} and the structures of $G/\gamma_2(G)$ and $\Z(G)$. This completes the proof of the theorem. \hfill $\Box$

\section{Groups $G$ with $\Aut(G)$ elementary abelian}

In this section we construct $p$-groups with elementary abelian automorphism group.
We start with the following result, which provides some structural information of a group $G$ for which $\Aut(G)$ is elementary abelian.
 
\begin{thm}\label{thm1BL}
 Let $G$ be a finite $p$-group such that $\Aut(G)$ is elementary abelian, where $p$ is an odd prime. Then one of the following two conditions holds true:
\begin{enumerate}
 \item $\Z(G) = \Phi(G)$ is elementary abelian;
\item  $\gamma_2(G) = \Phi(G)$ is elementary abelian.
\end{enumerate}
Moreover, the exponent of $G$ is $p^2$.
\end{thm}
\begin{proof}
 Since $\Aut(G)$ is elementary abelian, $G/Z(G)$ is elementary abelian and it follows from a result of Hopkins \cite[Section 3]{cH27} that $G$ is purely non-abelian. It follows from Theorem \ref{thmJ} that either  $Z(G)$ or $G/\gamma_2(G)$ is of exponent $p$. If the exponent of $Z(G)$ is $p$, then $Z(G) \le \Phi(G)$. Indeed, if $\Phi(G) < \Z(G)$, then $G$ has a non-trivial abelian direct factor, which is not possible as $G$ is purely non-abelian. Hence  $\Z(G) = \Phi(G)$ is elementary abelian. If the exponent of $G/\gamma_2(G)$ is $p$, then obviously $\gamma_2(G) = \Phi(G)$. Since the exponent of $\gamma_2(G)$ is equal to the exponent of $G/\Z(G)$, it follows that $\gamma_2(G) = \Phi(G)$ is elementary abelian. In any case the exponent of $\Phi(G)$ is $p$. Thus the exponent of $G$ is at most $p^2$. That the exponent of $G$ can not be $p$, follows from Theorem \ref{thmE}. Hence the exponent of $G$ is $p^2$. This completes the proof of the theorem.
\hfill $\Box$

\end{proof}

 Now we proceed to construct $p$-group $G$ such that $\Phi(G) < Z(G)$, and $\Aut(G)$ and  $\gamma_2(G) = \Phi(G)$ are elementary abelian.  Let $p$ be any prime, even or odd. Consider the group
\begin{eqnarray} \label{group2}
 G_1 &=& \langle x_1,x_2,x_3,x_4,x_5 \mid x_1^{p^2}= x_2^{p^2}=x_3^{p^2}= x_4^{p^2}= x_5^{p} = 1, [x_1,x_2] = x_1^p,\\
    & &[x_1,x_3] = x_3^p, [x_1,x_4] = 1, [x_1,x_5] = x_1^p, [x_2,x_3] = x_2^p, [x_2,x_4] = 1,\nonumber\\  
    & &[x_2,x_5] = x_4^p, [x_3,x_4] = 1, [x_3,x_5] = x_4^p, [x_4, x_5] = 1 \rangle.\nonumber
\end{eqnarray}

It is easy to see the following properties of $G_1$.

\begin{lemma}\label{lemma2}
The group $G_1$ is a p-group having order $p^9$, $\gamma_2(G_1) = \Phi(G_1) < \Z(G_1)$, $\Phi(G_1)$ is elementary abelian and the exponent of $Z(G_1)$ is $p^2$, where $p$ is any prime. Moreover, if $p$ is odd, then $G_1$ is regular.
\end{lemma}

As mentioned in the introduction, it can be checked by using GAP that for $p = 2$, $\Aut(G_1)$ is elementary abelian. So we assume that $p$ is odd.
For $1 \le i \le 5$, let $e_{x_i} = x_1^{a_{i1}}x_2^{a_{i2}}x_3^{a_{i3}}x_4^{a_{i4}}x_5^{a_{i5}}$, where $a_{ij}$ are non-negative integers for $1\leq i,j\leq 5$. 
Let $\alpha$ be an arbitrary automorphism of $G_1$. Since the nilpotency class of $G_1$ is $2$ and $\gamma_2(G_1)$ is generated by the set $\{x_i^p \mid 1 \le i \le 4\}$,  we can write 
\begin{equation}\label{aut1}
 \alpha(x_i) = x_i \prod_{j=1}^5 x_j^{a_{ij}}
\end{equation}
 for some non-negative integers $a_{ij}$ for $1 \le i, j \le 5$. 

\begin{lemma}\label{lemmab}
Let $\alpha$ be the automorphism of $G_1$ defined in \eqref{aut1}. Then
\begin{equation}\label{eqn4table1}
 a_{4j} \equiv 0 \mod{p} \;\;\text{for} \;\; j = 1, 2, 3, 5.
\end{equation}
\end{lemma}
\begin{proof}
Since $x_4 \in Z(G_1)$,  it follows that $\alpha(x_4) = x_4^{1+a_{44}}x_1^{a_{41}}x_2^{a_{42}} x_3^{a_{43}}x_5^{a_{45}} \in Z(G_1)$. This is possible only when $a_{4j} \equiv 0 \mod{p}$ for $j = 1, 2, 3, 5$, which completes the proof of the lemma.
\hfill $\Box$

\end{proof}

We'll make use of the following table in the proof of Theorem B, which is produced in the following way. The equation in the $k$th row is obtained by applying $\alpha$ on the relation in $k$th row, then comparing the powers of $x_i$ in the same row, and using preceding equations in the table and equations \eqref{eqn4table1}. For example, equation in $5$th row is obtained by applying $\alpha$ on $[x_1,x_3] = x_3^p$, then comparing the powers of $x_2$ and using equations in $2$nd and $3$rd row.

\begin{table}[h]
\centering
\begin{tabular}{|c||c|c|c|}
\hline
No. & equations & relations         & $x_i$'s  \\
\hline
\hline
1 & $a_{5j}\equiv 0 \mod{p}$, 1$\leq j\leq$ 4 & $x_5^{p} = 1$ & $x_1,x_2,x_3,x_4$\\
\hline
2 &$a_{12}  \equiv 0 \mod{p}$ & $[x_1,x_5] = x_1^p$ & $x_2$\\
\hline
3 & $a_{13}  \equiv 0 \mod{p}$ & $[x_1,x_5] = x_1^p$ & $x_3$\\
\hline
4 & $a_{14}  \equiv 0 \mod{p}$ & $[x_1,x_5] = x_1^p$ & $x_4$\\
\hline
5 & $a_{32}  \equiv 0 \mod{p}$ & $[x_1,x_3] = x_3^p$ & $x_2$\\
\hline
6 & $a_{55}(1+a_{11})\equiv 0 \mod{p}$ & $[x_1,x_5] = x_1^p$ & $x_1$\\
\hline
7 & $a_{23}(1+a_{11})\equiv 0 \mod{p}$ & $[x_1,x_2] = x_1^p$ & $x_3$\\
\hline
8 & $a_{21}+a_{21}a_{55} \equiv 0 \mod{p}$ & $[x_2,x_5] = x_4^p$ & $x_1$\\
\hline
9 & $a_{31}+a_{31}a_{55} \equiv 0 \mod{p}$ & $[x_3,x_5] = x_4^p$ & $x_1$\\
\hline
10 & $a_{35} + a_{11}a_{35} - a_{15}a_{31}- a_{31} \equiv 0 \mod{p}$ & $[x_1,x_3] = x_3^p$ & $x_1$\\
\hline
11 & $a_{11}(1+a_{33}) \equiv 0 \mod{p}$ & $[x_1,x_3] = x_3^p$ & $x_3$\\
\hline
12 & $a_{33}(1+a_{22}) \equiv 0 \mod{p}$ & $[x_2,x_3] = x_2^p$ & $x_2$\\
\hline
13 & $a_{55} + a_{33} + a_{33}a_{55} - a_{44} \equiv 0 \mod{p}$ & $[x_3,x_5] = x_4^p$ & $x_4$\\
\hline
14 & $a_{55} + a_{22} + a_{22}a_{55} + a_{23}(1 + a_{55}) - a_{44} \equiv 0 \mod{p}$ & $[x_2,x_5] = x_4^p$ & $x_4$\\
\hline
15 & $(a_{22}+a_{25})(1+ a_{11})- a_{15}a_{21} \equiv 0 \mod{p}$ & $[x_1,x_2] = x_1^p$ & $x_1$\\
\hline
16 & $a_{35}(1 + a_{23} + a_{22})- a_{25}( 1 + a_{33}) - a_{24} \equiv 0 \mod{p}$ & $[x_2,x_3] = x_2^p$ & $x_4$\\
\hline
17 & $-a_{15} - a_{15}a_{22} - a_{15}a_{23} \equiv 0 \mod{p}$ & $[x_1,x_2] = x_1^p$ & $x_4$\\
\hline
18 & $- a_{15}- a_{15}a_{33}- a_{34} \equiv 0 \mod{p}$ & $[x_1,x_3] = x_3^p$ & $x_4$\\
\hline
\end{tabular}
\caption{Table for the group $G_1$}
\end{table}

Now we are ready to prove Theorem B stated in the introduction. In the following proof, by (k) we mean the equation in the $k$th row of Table 1.
\vspace{.1in}

\noindent{\bf Proof of Theorem B.} It follows from Lemma \ref{lemma2} that $G_1$ is of order $p^9$, $\Phi(G_1) < \Z(G_1)$ and $\gamma_2(G_1) = \Phi(G_1)$ is elementary abelian. It is easy to show that the order of $\Autcent(G_1)$ is $p^{20}$. Since $\Aut(G_1)$ is elementary abelian for $p=2$, assume that $p$ is odd. We now prove that all automorphisms of $G_1$ are central.  
Let $\alpha$ be the automorphism of $G_1$ defined in \eqref{aut1}, i.e., $\alpha(x_i) = x_i\prod_{j=1}^5 x_i^{a_{ij}}$, where  $a_{ij}$ are non-negative integers for $1 \leq i, j \leq 5$.  Since $G_1/\Z(G_1)$ is elementary abelian, it is sufficient  to prove that $a_{ij} \equiv 0 \mod{p}$ for $1 \leq i, j \leq 5$.

Since $\alpha (x_1^p) = x_1^{p(1 + a_{11})}\prod_{j=2}^5 x_j^{pa_{1j}} \neq 1$, $x_5^p = 1$ and $a_{1j} \equiv 0 \mod{p}$ for $2 \le j \le 4$, it follows that $1 + a_{11}$ is not divisible by $p$. Therefore (6) and (7) give $a_{55} \equiv 0$ mod $p$ and $a_{23}  \equiv 0 \mod{p}$ respectively. Thus by (8) and (9) respectively, we get $a_{21} \equiv 0 \mod{p}$ and $a_{31}  \equiv 0 \mod{p}$. Using the fact that $a_{31}  \equiv 0 \mod{p}$, (10) reduces to the equation $a_{35}(1+ a_{11}) \equiv 0 \mod{p}$. Since  $1 + a_{11}$ is not divisible by $p$, we get  $a_{35}\equiv 0 \mod{p}$. Observe that $1 + a_{33}$ is not divisible by $p$. For, suppose $p$ divides $1 + a_{33}$. Since $a_{31}, a_{32}, a_{35}$ divisible by $p$ and $x_4  \in Z(G_1)$, it follows that $\alpha(x_3) \in Z(G_1)$, which is not true. Using this fact, it follows from (11) that $a_{11}  \equiv 0 \mod{p}$. Using above information, (13), (14) and (15) reduces, respectively, to the following equations.
\begin{eqnarray}
a_{33} - a_{44}  &\equiv 0 \mod{p},\label{eqnthmb1}\\
a_{22} - a_{44}  &\equiv 0 \mod{p},\label{eqnthmb2}\\
a_{22} + a_{25}  &\equiv 0 \mod{p}.\label{eqnthmb3}
\end{eqnarray}
 Subtracting equation \eqref{eqnthmb2} from equation \eqref{eqnthmb1}, we get  $a_{33} - a_{22}  \equiv 0 \mod{p}$. Adding this to equation \eqref{eqnthmb3} gives  $a_{33}+ a_{25}  \equiv 0 \mod{p}$. Using this fact after adding (12) to equation \eqref{eqnthmb3}, we get $a_{22}(1+a_{33}) \equiv 0 \mod{p}$. Since $1 + a_{33}$ is not divisible by $p$,  $a_{22}  \equiv 0 \mod{p}$. Thus equations \eqref{eqnthmb2} and \eqref{eqnthmb3} give $a_{44}  \equiv 0 \mod{p}$ and $a_{25}  \equiv 0 \mod{p}$ respectively. So $a_{33}  \equiv 0 
\mod{p}$ from equation \ref{eqnthmb1}.  Now (16) and (17) give $a_{24}  \equiv 0 \mod{p}$ and $a_{15}  \equiv 0 \mod{p}$ respectively. Finally, from equation (18) we get $a_{34}  \equiv 0 \mod{p}$. Hence  all ${a_{ij}}'s$ are divisible by $p$, which shows that $\alpha$ is a central automorphism of $G_1$. Since $\alpha$ was an arbitrary automorphism of $G_1$, we get $\Aut(G_1) = \Autcent(G_1)$.

It now remains to prove that $\Aut(G_1)$ is elementary abelian. Notice that $G_1$ is purely non-abelian. Since $\gamma_2(G_1) = \Phi(G_1)$, the exponent of $G_1/\gamma_2(G_1)$ is $p$. That $\Aut(G_1) = \Autcent(G_1)$ is elementary abelian now follows from Theorem \ref{thmJ}. This completes the proof of the theorem.
\hfill $\Box$

\vspace{.1in}

Finally we proceed to construct a finite $p$-group $G$ such that $\Aut(G)$ is elementary abelian, $\gamma_2(G) < \Phi(G)$ and $\Phi(G) = \Z(G)$ is elementary abelian. Let $p$ be any prime, even or odd. Define the group

\begin{eqnarray} \label{group3}
 G_2 &=& \langle x_1,x_2,x_3,x_4 \mid x_1^{p^2}= x_2^{p^2}=x_3^{p^2}= x_4^{p^2} = 1, [x_1,x_2] = 1, \\
    & &[x_1,x_3] = x_4^p, [x_1,x_4] = x_4^p,  [x_2,x_3] = x_1^p, [x_2,x_4] = x_2^p, [x_3,x_4] = x_4^p\rangle.\nonumber
\end{eqnarray}

It is easy to prove the following lemma.

\begin{lemma}\label{lemma3}
The group $G_2$ is a $p$-group of order $p^8$, $\gamma_2(G_2) < \Phi(G_2)$ and $\Z(G_2) = \Phi(G_2)$ is elementary abelian, where $p$ is any prime. Moreover, if $p$ is odd, then $G_2$ is regular.
\end{lemma}

Again, as mentioned in the introduction, it can be checked by using GAP that for $p = 2$, $\Aut(G_2)$ is elementary abelian. So from now onwards, we assume that $p$ is odd.
For $1 \le i \le 4$, let $e_{x_i} = x_1^{a_{i1}}x_2^{a_{i2}}x_3^{a_{i3}}x_4^{a_{i4}}$, where $a_{ij}$ are non-negative integers for $1\leq i,j\leq 4$.  Let $\alpha$ be an arbitrary automorphism of $G_2$. Since the nilpotency class of $G_2$ is $2$ and $\gamma_2(G_2)$ is generated by the set $\{x_1^p, x_2^p, x_4^p\}$,  we can write 
\begin{equation}\label{aut2}
 \alpha(x_i) = x_i \prod_{j=1}^4 x_j^{a_{ij}}
\end{equation}
 for some non-negative integers $a_{ij}$ for $1 \le i, j \le 4$.

The following table, which will be used in the proof of Theorem C below, is produced in a similar fashion as Table 1.

\begin{table}[h]
\centering
\begin{tabular}{|c||c|c|c|}
\hline
No. & equations & relations         & $x_i$'s  \\
\hline
\hline
1 & $a_{13}\equiv 0 \mod{p}$ & $[x_2,x_3] = x_1^p$ & $x_3$\\
\hline
2 & $a_{23}\equiv 0 \mod{p}$ & $[x_2,x_4] = x_2^p$ & $x_3$\\
\hline
3 & $a_{43}\equiv 0 \mod{p}$ & $[x_1,x_3] = x_4^p$ & $x_3$\\
\hline
4 & $a_{41}\equiv 0 \mod{p}$ & $[x_1,x_4] = x_4^p$ & $x_1$\\
\hline
5 & $a_{21}\equiv 0 \mod{p}$ & $[x_2,x_4] = x_2^p$ & $x_1$\\
\hline
6 & $a_{24}\equiv 0 \mod{p}$ & $[x_2,x_4] = x_2^p$ & $x_4$\\
\hline
7 & $a_{14}\equiv 0 \mod{p}$ & $[x_2,x_3] = x_1^p$ & $x_4$\\
\hline
8 & $a_{44}(1+a_{22}) \equiv 0 \mod{p}$ & $[x_2,x_4] = x_2^p$ & $x_2$\\
\hline
9 & $a_{11} + a_{11}a_{44} \equiv 0 \mod{p}$ & $[x_1,x_4] = x_4^p$ & $x_4$\\
\hline
10 & $a_{22} + a_{22}a_{33} + a_{33} - a_{11} \equiv 0 \mod{p}$ & $[x_2,x_3] = x_1^p$ & $x_1$\\
\hline
11 & $-a_{42}(1+a_{33}) \equiv 0 \mod{p}$ & $[x_3,x_4] = x_4^p$ & $x_1$\\
\hline
12 & $a_{12} + a_{12}a_{44} - a_{42} \equiv 0 \mod{p}$ & $[x_1,x_4] = x_4^p$ & $x_2$\\
\hline
13 & $a_{32} + a_{32}a_{44} - a_{34}a_{42} - a_{42} \equiv 0 \mod{p}$ & $[x_3,x_4] = x_4^p$ & $x_2$\\
\hline
14 & $a_{34}(1+a_{22}) - a_{12} \equiv 0 \mod{p}$ & $[x_2,x_3] = x_1^p$ & $x_2$\\
\hline
15 & $ a_{11}(1 +  a_{33} + a_{34}) + a_{33} + a_{34} - a_{44} \equiv 0 \mod{p}$ & $[x_1,x_3] = x_4^p$ & $x_4$\\
\hline
16 & $a_{31}+  a_{31}a_{44}+ a_{33}+ a_{33}a_{44} \equiv 0 \mod{p}$ & $[x_3,x_4] = x_4^p$ & $x_4$\\
\hline
\end{tabular}
\caption{Table for the group $G_2$}
\end{table}

Now we are ready to prove Theorem C stated in the introduction. By (k), in the following proof, we mean the equation in the $k$th row of Table 2.
\vspace{.1in}
 
\noindent{\bf Proof of Theorem C.}
It follows from Lemma \ref{lemma3} that $G_2$ is of order $p^8$, $\gamma_2(G_2) < \Phi(G_2)$ and $\Z(G_2) = \Phi(G_2)$ is elementary abelian.  It is again easy to show that the order of $\Autcent(G_2)$ is $p^{16}$. Since $\Aut(G_2)$ is elementary abelian for $p=2$, assume that $p$ is odd. As in the proof of Theorem B, to show that all automorphisms of $G_2$ are central, it is sufficient to show that $a_{ij} \equiv 0 \mod{p}$ for $ 1 \le i, j \le 4$.

Since $a_{21} , a_{23}, a_{24}$ are divisible by $p$, it follows that $(1 + a_{22})$ is not divisible by $p$. For, if $p$ divides $(1 + a_{22})$, then $\alpha(x_2) \in Z(G_2)$, which is not possible. Using this fact, (8) gives $a_{44} \equiv 0 \mod{p}$. Thus from (9) we get $a_{11} \equiv 0 \mod{p}$. Now we observe from (10) that $(1 + a_{33})$ is not divisible by $p$. For, suppose, $(1 + a_{33})$ is divisible by $p$, then using the fact that $a_{11} \equiv 0 \mod{p}$, (10) gives $a_{33} \equiv 0 \mod{p}$, which is not possible. Thus (11) gives $a_{42} \equiv 0 \mod{p}.$ Now using that $a_{42}$ and $a_{44}$ are divisible by $p$,  (12) and (13) give $a_{12} \equiv 0 \mod{p}$ and $a_{32} \equiv 0 \mod{p}$ respectively. Since $a_{12} \equiv 0 \mod{p}$ and $(1 + a_{22})$ is not divisible by $p$, (14) gives $a_{34} \equiv 0 \mod{p}.$ Using that $a_{11}$, $a_{34}$ and $a_{44}$ are divisible by $p$,  (15) gives $a_{33} \equiv 0 \mod{p}$. Now using that $a_{33}$ and $a_{44}$ are divisible by $p$, equation (16) gives $a_{31} \equiv 0 \mod{p}$. Since $a_{11}$ and $a_{33}$ are divisible by $p$,  equation (10) gives $a_{22} \equiv 0 \mod{p}$. Hence $a_{ij} \equiv 0 \mod{p}$ for $1 \le i, j \le 4$.

Since $\Z(G_2)$ is elementary abelian, $\Aut(G_2) = \Autcent(G_2)$ is elementary abelian by Theorem \ref{thmJ}. This completes the proof of the theorem. \hfill $\Box$

\noindent{\bf Acknowledgements.} Authors thank Prof. Mike Newman  for his useful comments and suggestions.

\end{document}